\documentclass[12pt, reqno]{amsart}

\usepackage{amssymb, amsmath, amsthm}
\usepackage[backref]{hyperref}
\hypersetup{
  colorlinks   = true,          
  urlcolor     = blue,          
  linkcolor    = blue,          
  citecolor   = red             
}

\usepackage[alphabetic,backrefs,lite]{amsrefs}
\usepackage{amscd}   
\usepackage{fullpage}
\usepackage[all]{xy} 
\usepackage{microtype}
\usepackage{enumitem}
\setlist[enumerate]{label=(\roman*)}

\usepackage{tikz}									
\usetikzlibrary{matrix}
\usetikzlibrary{patterns}

\usepackage{color}		

\newtheorem{lem}{Lemma}[section]

\newtheorem{prop}[lem]{Proposition}

\newtheorem{claim*}{Claim}
\newtheorem{rmk}[lem]{Remark}
\newtheorem{thm}[lem]{Theorem}

\theoremstyle{definition}
\newtheorem{rem}[lem]{Remark}
\newtheorem{defn}[lem]{Definition}
\newtheorem{ex}[lem]{Example}

\newcommand{\G}{{\mathbb G}}

\newcommand{\PP}{{\mathbb P}}
\newcommand{\C}{{\mathbb C}}

\newcommand{\Q}{{\mathbb Q}}
\newcommand{\R}{{\mathbb R}}
\newcommand{\Z}{{\mathbb Z}}
\newcommand{\N}{{\mathbb N}}

\newcommand{\calC}{{\mathcal C}}

\newcommand{\calO}{{\mathcal O}}

\newcommand{\calT}{{\mathcal T}}

\newcommand{\frakU}{{\mathfrak U}}
\newcommand{\frakX}{{\mathfrak X}}
\newcommand{\frakC}{{\mathfrak C}}

\DeclareMathOperator{\HH}{H}

\DeclareMathOperator{\Hom}{Hom}

\DeclareMathOperator{\Spec}{Spec}

\DeclareMathOperator{\trop}{trop}

\DeclareMathOperator{\Trop}{Trop}

\DeclareMathOperator{\val}{val}
\DeclareMathOperator{\ob}{ob}
\DeclareMathOperator{\indeg}{in}	

\newcommand{\isom}{\cong}

\newtheorem*{question*}{Question}
\numberwithin{equation}{section}
\numberwithin{table}{section}

\newcommand{\defi}[1]{\emph{#1}} 

\title{Faithful realizability of tropical curves}

\author{Man-Wai Cheung}
\curraddr{Department of Mathematics, University of California, San Diego, La Jolla, CA 92093-0112, USA}
\email{m1cheung@ucsd.edu}
\urladdr{http://www.math.ucsd.edu/~m1cheung/}

\author{Lorenzo Fantini}
\curraddr{Centre de Math\'ematiques Laurent Schwartz de l'\'Ecole Polytechnique, 91128 Palaiseau, France}
\email{lorenzo.fantini@polytechnique.edu}
\urladdr{www.math.polytechnique.fr/~fantini/}

\author{Jennifer Park}
\curraddr{Department of Mathematics, McGill University, Montr\'eal, Qu\'ebec, Canada}
\email{jennifer.park2@mcgill.ca}
\urladdr{www.math.mcgill.ca/jpark/}

\author{Martin Ulirsch}
\curraddr{Department of Mathematics, Brown University, Providence, RI 02912, USA}
\email{ulirsch@math.brown.edu}
\urladdr{http://www.math.brown.edu/~ulirsch/}

\date{\today}

\begin{document}

\begin{abstract}
We study whether a given tropical curve $\Gamma$ in $\mathbb{R}^n$ can be realized as the tropicalization of an algebraic curve whose non-archimedean skeleton is faithfully represented by $\Gamma$.
We give an affirmative answer to this question for a large class of tropical curves that includes all trivalent tropical curves, but also many tropical curves of higher valence.
We then deduce that for every metric graph $G$ with rational edge lengths there exists a smooth algebraic curve in a toric variety whose analytification has skeleton $G$, and the corresponding tropicalization is faithful.
Our approach is based on a combination of the theory of toric schemes over discrete valuation rings and logarithmically smooth deformation theory, expanding on a framework introduced by Nishinou and Siebert. 
\end{abstract}

\maketitle

\section{Introduction}

\subsection{} Let $N$ be a finitely generated free abelian group, write $M=\Hom(N,\Z)$ for its dual, and set $N_{\R} := N \otimes_{\Z} \R$.
Let $K$ be a non-archimedean field, let $T$ be the split algebraic torus $\Spec K[M]$, and let $T^{\textup{an}}$ be the non-archimedean analytic space associated to $T$ in the sense of \cite{Berkovich90}. 
Following \cite{Gubler07}, \cite{Gubler13}, and \cite{EKL}, one can define a continuous tropicalization map 
$\trop\mathrel{\mathop:}T^{an}\longrightarrow N_\R$. 
%
Given a curve $C$ in $T$, its \emph{tropicalization} $\Trop(C)$ is the subset of $N_\R$ defined by
$\Trop(C):=\trop(C^{an})$.
 By the Bieri-Groves Theorem \cite{BieriGroves84}*{Theorem A} and \cite{EKL}*{Theorem 2.2.3} the set $\Trop(C)$ can be canonically endowed with the structure of a 1-dimensional rational polyhedral complex.
By using the lattice length on $N_\R$ with respect to $N$, $\Trop(C)$ can also be seen as a metric graph, with some non-compact edges which have infinite length.
Moreover, we can associate weights to the edges of $\Trop(C)$  which satisfy a natural \emph{balancing condition}.

Abstracting from these properties and following \cite{Mikhalkin05}, we define a \emph{tropical curve} as a balanced weighted $1$-dimensional rational polyhedral complex; 
see Definition \ref{D: tropicalcurve} for a precise definition. 
It is then natural to ask which tropical curves can be realized as the tropicalization of an algebraic curve in $T$; this is known as the problem of \emph{realizability} of tropical curves.
For the basics of tropical geometry we refer the reader to \cite{Mikhalkin05}, \cite{Gubler13} and \cite{MaclaganSturmfels14}.

\subsection{} Let us now assume that $K$ is nontrivially valued, and denote by $R$ its valuation ring. 
Let $C$ be a smooth, complete, and connected curve over $K$.
While the underlying topological space of the non-archimedean analytic curve $C^{an}$ is an infinite graph, Berkovich shows in \cite{Berkovich90}*{Section 4.3} that it has the homotopy type of a finite graph. 
More precisely, he associates to every semistable $R$-model $\calC$ of $C$, i.e. a flat and proper $R$-scheme $\calC$ with generic fiber $C$ and nodal special fiber $\calC_s$, a subset $\Sigma_\calC$ of $C^{an}$, called a \emph{skeleton}, and shows that it is a deformation retract of $C^{an}$. 
Note that a semistable model $\calC$ of $C$ always exists, possibly after a finite separable base change, by the semistable reduction theorem \cite{DeligneMumford69}.
As an abstract graph, the skeleton $\Sigma_C$ is the dual graph of the special fiber $\calC_s$ of $\calC$, and it can be naturally endowed with the structure of a metric graph. 


Let $K^{\textup{alg}}$ denote a fixed algebraic closure of $K$. 
Given a finite set $V$ of $K^{\textup{alg}}$-points of $C$, we can associate an enlarged skeleton $\Sigma_{\calC, V}$ which contains $\Sigma_\calC$ and has new edges of infinite length corresponding to the points of $V$.
Again, there is an embedding of this enlarged skeleton as a deformation retract $\Sigma_{\calC,V}$ of $C^{an}$.
We refer the reader to \cite{BakerPayneRabinoff14}*{Section 3} for the details of this construction.

\subsection{} Given an embedding of $C$ into a toric variety $X$ with big torus $T$, we say that the corresponding tropicalization is \emph{faithful} with respect to a skeleton $\Sigma_{\calC,V}$ if $\trop$ induces an isometric homeomorphism from $\Sigma_{\calC,V}\cap T^{an}$ onto its image in $\Trop(C\cap T)$.
In addition, if $\Sigma_{\calC,V}\cap T^{an}$ surjects onto $\Trop(C\cap T)$, we say that the tropicalization is \emph{totally faithful} with respect to $\Sigma_{\calC,V}$. Examples in \cite{BPR11}*{Section 2.5} show that many tropicalizations are not faithful.
In this article we study the following refinement of the question of realizability:

\begin{question*}[Faithful realizability]
\label{question_faithfulrealizability}
Given a tropical curve $\Gamma\subseteq N_\R$, can we find a smooth, complete and connected curve $C$ over $K$, a skeleton $\Sigma_{\calC,V}$ of $C^{an}$, and an embedding of $C$ into a toric variety $X$ with dense torus $T$ such that 
\begin{equation*}
\Trop(C\cap T)=\Gamma
\end{equation*} and the tropicalization is totally faithful with respect to $\Sigma_{\calC,V}$? 
\end{question*}

This question is also related to the problem of faithful tropicalization investigated in \cite{BPR11} for curves and in \cite{GublerRabinoffWerner14} in higher dimension.


\subsection{}

We give a positive answer to both parts of the question of faithful realizability, working over a finite extension $\C((t^{1/\ell}))$ of $\C((t))$, when $\Gamma\subseteq N_\R$ is a tropical curve with rational edge lengths and fulfills the following three conditions:

\begin{itemize}
\item\label{condition_nonsuperabundant} $\Gamma$ is \emph{non-superabundant} (Definition \ref{D: abundancy}).
\item\label{condition_smooth} $\Gamma$ is \emph{smooth} (Definition \ref{definition_smooth_tropical_curve}).
\item\label{condition_valency} $\Gamma$ is \emph{3-colorable} (Definition \ref{definition_tame_tropical_curve}).
\end{itemize}

Non-superabundancy for tropical curves is a natural genericity condition, originally introduced in \cite{Mikhalkin05}, roughly saying that the deformation space of $\Gamma$ as an embedded metric graph in $N_\R$ has the expected dimension (see the paragraph after Definition \ref{D: abundancy}); in our situation we adopt the point of view of \cite{Katz12lifting}.
The notion of 3-colorability is a mild combinatorial condition which is always satisfied by trivalent tropical curves and tropical curves of genus at most three.
More precisely we prove the following result:

\begin{thm}\label{T: one}
Let $\Gamma$ be a non-superabundant, smooth and 3-colorable tropical curve with rational edge lengths. Then there exists
\begin{itemize} 
\item a finite extension $K=\C((t^{1/\ell}))$ of $\C((t))$ with valuation ring $R$,
\item a toric scheme $\frakX$ over $R$ with big torus $T$, 
\item a complete smooth curve $C$ over $K$, 
\item a semistable $R$-model $\calC$ of $C$ together with an embedding of $\calC$ into $\frakX$, and
\item a finite set $V\subseteq C(K^{alg})$ of marked points, 
\end{itemize}
such that 
\begin{equation*}
\Trop(C\cap T)=\Gamma
\end{equation*}
and the tropicalization is totally faithful with respect to $\Sigma_{\calC,V}$. 
\end{thm}

Now we can combine Theorem \ref{T: one} with a result of \cite{CDMY}, which states the following: 
Given a metric graph $G$ with rational edge lengths, there is a non-superabundant, smooth and 3-colorable tropical curve $\Gamma$ in $\R^n$ whose skeleton is $G$. We will recall their result and adjust it to our needs in Theorem \ref{thm: combinatorics}. This leads to the following theorem:

\begin{thm}\label{T: two}
For every metric graph $G$ with rational edge lengths there exists  
\begin{itemize}
\item a tropical curve $\Gamma$ in $\R^n$, where
\begin{equation*}
n=\max\big\{3,\max\{\deg v-1\vert v\in E(G)\}\big\} \ ,
\end{equation*} 
\item a finite extension $K=\C((t^{1/\ell}))$ of $\C((t))$ with valuation ring $R$,
\item a toric scheme $\frakX$ over $R$ with big torus $T=\G_m^n$, 
\item a complete smooth curve $C$ over $K$,
\item a semistable $R$-model $\calC$ of $C$ together with an embedding of $\calC$ into $\frakX$, and
\item a finite set $V\subseteq C\big(K^{alg})$ of marked points, 
\end{itemize}
such that 
\begin{equation*}
\Trop(C\cap T)=\Gamma \ ,
\end{equation*}
the skeleton $\Sigma_\calC$ 
is equal to $G$, and the tropicalization is totally faithful with respect to $\Sigma_{\calC,V}$. 
\end{thm}

In particular, if $G$ has no univalent vertices then the skeleton $\Sigma_\calC$ in Theorem \ref{T: two} is the minimal skeleton of $C^{an}$, since each proper subgraph of $G$ has homotopy type different than the homotopy type of $C^{an}$. The methods of \cite{CDMY} allow us to consider more general graphs $G$ that may have infinite rays. Our proof applies to this situation, giving us the same data as in Theorem \ref{T: two} together with an additional subset  $V'\subseteq V$ of marked points such that $G=\Sigma_{\calC,V'}$.

\subsection{} Our proof of Theorem \ref{T: one} generalizes the methods developed in \cite{NishinouSiebert06} and extended in \cite{Nis141}.
We now outline the steps of the proof: Let $\Gamma\subset N_\R$ be a non-superabundant, smooth and 3-colorable tropical curve with rational edge lengths. Choose a finite extension $K=\C((t^{{1}/{\ell}}))$ of $\C((t))$ such that all vertices of $\Gamma$ have coordinates in the value group of $K$. Define $\Delta$ to be the fan in $N_\R\times\R_{\geq 0}$ obtained by putting a copy of $\Gamma$ in the height one part $N_\R\times\{1\}$ of $N_\R\times\R_{\geq 0}$ and taking cones over all edges and vertices of $\Gamma$. The fan $\Delta$ defines a toric scheme $\frakX=\frakX_\Delta$ over $R=\C[[t^{\frac{1}{\ell}}]]$.

\begin{enumerate}[label=(\Roman*)]
 \item\label{proof_constructC0} In Section \ref{section_specialfiber} we define a suitable nodal curve $C_0$ in the special fiber $\frakX_s$ of $\frakX$, whose dual graph is the skeleton of $\Gamma$.  
\item\label{proof_liftC0} The special fiber $\frakX_s$ is log smooth over the standard log point $O_0=(\Spec\C,\N)$; and the curve $C_0$, endowed with the log structure induced from $\frakX_s$, is log smooth over $O_0$. These observations allow us to apply log smooth deformation theory in Section \ref{section_logdef} and Section \ref{section_abundancy} in order to show that we can lift the nodal curve $C_0$ to a proper, flat semistable curve $\calC\subseteq \frakX$ over $R$ with special fiber $C_0$. The generic fiber $C$ of $\calC$ is a smooth complete curve in the generic fiber $X$ of $\frakX$, which is a $T$-toric variety over $K$. We set $V=C\cap (X-T)$.
\item Finally, in Section \ref{section_proofs}, we verify that this construction satisfies the properties asserted in Theorem \ref{T: one}: Lemma \ref{lemma_trop=Gamma} shows that in this case $\Trop(C\cap T)=\Gamma$, and the results of \cite{BPR11}*{Section 6} allow us to show that $\trop\vert_
{\Sigma_{\calC,V}}:\Sigma_{\calC,V}\rightarrow \Trop(C)$ is totally faithful.
\end{enumerate}

The crucial technical insight of this paper is contained in Section \ref{section_abundancy}, where we give a combinatorial interpretation to the homomorphism of cohomology groups on $C_0$ controlling the logarithmic deformation theory used in Step \ref{proof_liftC0} of our proof. In particular, we deduce that if $\Gamma$ is non-superabundant, this homomorphism is surjective and the deformations are unobstructed.

Smoothness and 3-colorability of the tropical curve $\Gamma$ are only used in Step \ref{proof_constructC0} of our proof. The proof would work equally well with any other condition on $\Gamma$ that allows us to construct a suitable nodal curve $C_0$ in $\frakX_s$. The classical condition of trivalency used extensively in the literature implies 3-colorability; in particular, Theorem 5.1.7 holds for all smooth non-superabundant trivalent tropical curves. 

\subsection{} The question of realizability of tropical curves has initially been studied by Mikhalkin \cite{Mikhalkin05} in order to count algebraic curves in toric surfaces. Nishinou and Siebert \cite{NishinouSiebert06} generalize his results to loopless (hence non-superabundant) trivalent tropical curves in higher dimension, and Tyomkin \cite{Tyomkin12} further generalizes this to non-superabundant trivalent tropical curves of higher genus. See \cite{Gross_Kansaslectures}*{Chapter 4} for an expository account of the methods of \cite{NishinouSiebert06}.
Special conditions ensuring the realizability of some superabundant tropical curves are studied by Speyer in \cite{SpeyerThesis}, and by Katz in \cite{Katz12lifting}. 
In the upcoming \cite{Nis141} Nishinou extends the approach of \cite{NishinouSiebert06} to more general trivalent tropical curves (see also \cite{Nishinou15}).
Another approach to the problem in the trivalent case has recently been developed by Lang in \cite{Lang15}.
Our Theorem~\ref{T: one} gives realizations for a wide class of tropical curves, including tropical curves of higher valence and higher genus.

\subsection*{Acknowledgements} This paper arose as a group project at the AMS 2013 Mathematics Research Communities workshop on Tropical and Non-Archimedean Geometry in Snowbird, Utah. The authors would like to thank the American Mathematical Society, and in particular Ellen Maycock, for their support of this program, as well as the organizers Matt Baker and Sam Payne. Particular thanks are due to Sam Payne for his advice and encouragement, and for suggesting us to work on this project in the first place.
We thank Dustin Cartwright, Andrew Dudzik, Madhusudan Manjunath, and Yuan Yao for sharing early versions of \cite{CDMY}; we are particularly grateful to Dustin Cartwright for giving us the key insight on how to adjust their result to our needs in Theorem \ref{thm: combinatorics} and for several remarks on an early version of this article. We also thank Takeo Nishinou for sharing a preliminary version of his upcoming paper \cite{Nis141}.
Parts of this paper have been written while enjoying the hospitality of Yale University, for which particular thanks are due to Sam Payne. The first author would also like to thank the University of Cambridge for their hospitality. The fourth author would like to thank the University of Leuven, and in particular Johannes Nicaise, for their hospitality. Finally, we thank Dan Abramovich, Mark Gross, Sam Payne, Bernd Siebert, and the three anonymous referees for remarks on an earlier version of this paper. We also profited from discussions with Eric Katz, Elisa Postinghel, and Dhruv Ranganathan.


\section{Construction of the special fiber}\label{section_specialfiber}

We begin by recalling the definitions of a tropical curve and stating the combinatorial conditions our tropical curves will have to satisfy.

\begin{defn}
\label{D: tropicalcurve}
Let $\Gamma$ be the (non-compact) weighted graph obtained by removing all vertices of valence one from a finite connected weighted graph.
We say that $\Gamma$ is a \emph{tropical curve} in $N_\R$ if it is endowed with a closed embedding $\Gamma\subset N_\R$ such that every vertex of $\Gamma$ is contained in $N_\Q$, every edge of $\Gamma$ is contained in a line of of rational slope of $N_{\R}$ and the \emph{balancing condition} is satisfied: for every vertex $v$ of $\Gamma$, if $e_1, \dots e_m$ are the edges of $\Gamma$ adjacent to $v$, $w_1, \dots w_m \in \N $ are their weights and $\vec{e}_1, \ldots, \vec{e}_m \in N$ are the primitive integer vectors from $v$ in the direction of the edges $e_j$, then we have $\sum_{j=1}^m w_j\vec{e}_j = 0$.
\end{defn}

Moreover, a tropical curve $\Gamma\subset N_\R$ can be endowed with the structure of metric graph by using the lattice length on $N_\R$ with respect to $N$.
That is, the length of an edge between the two vertices $v_1$ and $v_2$ is the largest positive real number $l$ such that $v_2-v_1=lv$ for some $v\in N$, and the unbounded edges of $\Gamma$ are then precisely those of infinite length.
We then call \emph{skeleton} of $\Gamma$ the finite metric graph obtained by erasing from $\Gamma$ all unbounded edges.

The following definition is the one-dimensional version of \cite{MikhalkinZharkov13}*{Definition 1.14}. For a planar tropical curve it also coincides with \cite{Mikhalkin05}*{Definition 2.18}.

\begin{defn}\label{definition_smooth_tropical_curve}
A tropical curve $\Gamma\subset N_\R$ is said to be \emph{smooth} if all its weights are equal to 1 and, for every vertex $v$ of $\Gamma$, there exists a basis $\{b_1, \ldots,b_n\}$ of $N$ and a positive integer $1\leq d\leq n$ such that the edges of $\Gamma$ adjacent to $x$ are in the directions of $\big\{b_1, \ldots, b_{d}, -\sum_{i=1}^{d}b_i\big\}$.
\end{defn}

We conclude our list of definitions with a mild combinatorial condition.

\begin{defn}\label{definition_tame_tropical_curve}
A tropical curve $\Gamma\subset N_\R$ is said to be \emph{3-colorable} if there exists an ordering $\{v_1,v_2,\ldots\}$ of the vertices of $\Gamma$ such that for every $i$, the vertex $v_i$ is adjacent to less than three of the vertices $\{v_j\}_{j<i}$.
\end{defn}

\begin{rmk}
A tropical curve $\Gamma$ is 3-colorable if and only if its coloring number (in the sense of \cite{ErdosHajnal66}) is at most three.
It is simple to show that $\Gamma$ is 3-colorable if and only if it is 2-degenerate in the sense of \cite{LickWhite70}, i.e. if and only if every subgraph $G$ of $\Gamma$ contains a vertex that is adjacent to at most two other vertices of $G$.
\end{rmk}

\begin{ex}
Every trivalent tropical curve is 3-colorable. 
Indeed, let $\Gamma$ be a trivalent tropical curve and let $m$ be the number of vertices of $\Gamma$.
The balancing condition implies that $\Gamma$ has at least one unbounded edge $e$; 
let $v_m$ be the vertex adjacent to $e$.
Then $v_m$ has bounded valence at most two, and after removing it and and all the adjacent edges we can find one new vertex $v_{m-1}$ of valence at most two. 
By repeating this argument until we have selected all the vertices of $\Gamma$ we obtain an ordering which satisfies the condition of Definition~ \ref{definition_tame_tropical_curve}, therefore $\Gamma$ is 3-colorable.
Observe that also any tropical curve of genus at most three is 3-colorable, since the smallest graph whose vertices have all bounded valence of three or more is a tetrahedron.
\end{ex}

Let $\Gamma$ be a tropical curve in $N_\R$. 
Let $K=\C((t^{1/\ell}))$ be a finite extension of $\C((t))$ such that all the vertices of the underlying graph of $\Gamma$ lie in $N_{v(K^{\times})}=N \otimes_{\Z}{v(K^{\times})}$, where $v(K^{\times})=\Z[\frac 1 {\ell}]$ is the value group of $K$, and let $R=\C\lbrack\lbrack t^{1/\ell}\rbrack\rbrack$ be the valuation ring of $K$. 
Define a set of cones in $N_\R\times\R_{\geq 0}$ by putting a copy of $\Gamma$ into $N_\R\times\{1\}$ and taking cones over all edges and vertices of $\Gamma$.
More precisely, let $\Delta$ be the collection of the cones
\begin{equation*}
c(F)=\overline{\R_{>0}(F\times\{1\})}\subseteq N_\R\times\R_{\geq 0}
\end{equation*}
and
\begin{equation*}
c(F)\cap(N_\R\times\{0\})\subseteq N_\R\times\R_{\geq 0},	
\end{equation*}
where $F$ is either an edge or a vertex of $\Gamma$.
Then $\Delta$ is a fan. 
Indeed, as in the proof of \cite{BurgosGilSombra12}*{Theorem 3.4} the only non-trivial part is to show that the set $\Delta_0=\Delta\cap (N_\R\times\{0\})$ of the recession cones of $\Gamma$ is a fan (see \cite{BurgosGilSombra12}*{Examples 3.1 and 3.9(i)}). 
Since $\Gamma$ is one-dimensional, the recession cones of $\Gamma$ are either rays starting at the origin or the origin itself. Two such rays are either equal or their intersection is the origin and this observation suffices to show that the recession cones form a fan.
Moreover, the fan $\Delta$ is $v(K^{\times})$-admissible in the sense of \cite{Gubler13}*{Definition 7.5}. 
We set $\frakX$ to be the toric scheme over $R$ defined by $\Delta$. 
\begin{rem} 
We could compactify $\frakX$ by completing the cone $\Delta$, but prefer not to do so, since we only want to keep the toric strata that are relevant to our construction.
\end{rem}

Using the correspondence of \cite{Gubler13}*{7.9} we can give an explicit description of the toric scheme $\frakX$. Its generic fiber $\frakX_\eta$ is the toric variety over $K$ associated to the fan $\Delta_0$.
Its special fiber $\frakX_s$ is reduced by \cite{Gubler13}*{Lemma 7.10}, since the valuation $v$ is discrete and the vertices of $\Gamma$ are in $N_{v(K^{\times})}$. The irreducible components of $\frakX_s$ are toric varieties over the residue field $\C$, and they correspond bijectively to the vertices of $\Gamma$. Whenever two vertices $v$ and $w$ are connected by an edge $e$, the two components $X_v$ and $X_w$ are glued along the boundary divisor corresponding to $e$.

If $\Gamma$ is smooth and $v$ is a vertex of $\Gamma$ with $\val_\Gamma(v)=d+1$, from the explicit description of $\Gamma$ around $v$ of Definition \ref{definition_smooth_tropical_curve}, we deduce that the corresponding component $X_v$ of $\frakX_s$ is isomorphic to $P_d\times \mathbb G_m^{n-d}$, where $P_d:=\mathbb P^d\setminus\{\mbox{orbits of $T$ of codimension 2 or higher}\}$, and the boundary divisors of $P_d$ are all isomorphic to $\mathbb G_m^{d-1}$. 

\begin{ex}\label{example_hexagonalthing}
Consider the admissible cone $\Delta\subseteq \R^2\times\R_{\geq 0}$ obtained as described above by placing at height 1 the following tropical curve $\Gamma\subseteq \R^2$ :
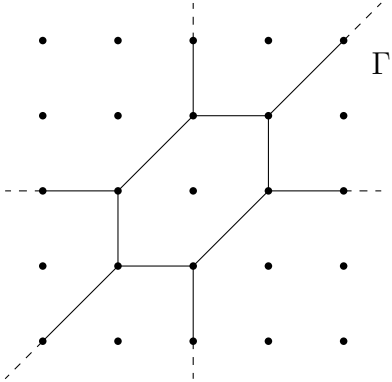
\vspace{1ex}
\begin{center}\begin{tikzpicture}
\foreach \a in {-1,0,1,2,3}
 \foreach \b in {-1,0,1,2,3}
  \fill (\a,\b) circle (0.05 cm);

\draw (0,0) -- (1,0);
\draw (0,0) -- (0,1);
\draw (0,1) -- (1,2);
\draw (1,2) -- (2,2);
\draw (1,0) -- (2,1);
\draw (2,1) -- (2,2);
\draw (-1,-1) -- (0,0);
\draw (-1,1) -- (0,1);
\draw[dashed] (-1,1) -- (-1.5,1);
\draw (1,-1) -- (1,0);
\draw[dashed] (1,-1) -- (1,-1.5);
\draw (1,2) -- (1,3);
\draw[dashed] (1,3) -- (1,3.5);
\draw (2,1) -- (3,1);
\draw[dashed] (3,1) -- (3.5,1);
\draw (2,2) -- (3,3);
\draw[dashed] (3,3) -- (3.5,3.5);
\draw[dashed] (-1,-1) -- (-1.5,-1.5);

\node at (3.5,2.7) {$\Gamma$};
\end{tikzpicture}\end{center}

We obtain a toric scheme $\frakX$ over $\C[[t]]$ whose generic fiber $X$ is the toric variety associated to the recession fan $\Delta_0=\Delta\cap( \R^2\times\{0\})$ of $\Gamma$, which is the following fan in $\mathbb{R}^2$:
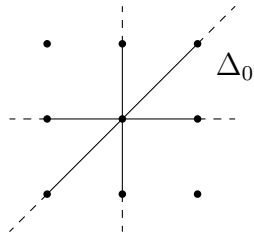
\vspace{1ex}
\begin{center}\begin{tikzpicture}
\foreach \a in {-1,0,1}
 \foreach \b in {-1,0,1}
  \fill (\a,\b) circle (0.05 cm);

\draw (0,0) -- (1,0);
\draw[dashed] (1,0) -- (1.5,0);
\draw (0,0) -- (0,1);
\draw[dashed] (0,1) -- (0,1.5);
\draw (0,0) -- (-1,0);
\draw[dashed] (-1,0) -- (-1.5,0);
\draw (0,0) -- (0,-1);
\draw[dashed] (0,-1) -- (0,-1.5);
\draw (0,0) -- (-1,-1);
\draw[dashed] (-1,-1) -- (-1.5,-1.5);
\draw (0,0) -- (1,1);
\draw[dashed] (1,1) -- (1.5,1.5);

\node at (1.5,0.7) {$\Delta_0$};

\end{tikzpicture}\end{center}
Therefore $X$ is the toric surface obtained by blowing up $\PP^2_{\C((t))}$ in three points, with the six closed points fixed by the torus removed. The special fiber $\frakX_s$ of $\frakX$ consists of six copies of $\PP^2_\C$ without their closed torus invariant points, glued over the one-dimensional toric strata as indicated by their moment polytopes below:
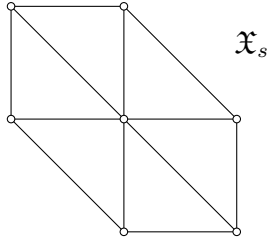
\vspace{1ex}
\begin{center}\begin{tikzpicture}

\foreach \a in {90,180,270,360}
\draw (0,0) -- (\a:1.5);
\draw (1.5,0) -- (0,1.5);
\draw (0,1.5) -- (-1.5,1.5);
\draw (-1.5,1.5) -- (-1.5,0);
\draw (-1.5,0) -- (0,-1.5);
\draw (0,-1.5) -- (1.5,-1.5);
\draw (1.5,-1.5) -- (1.5,0);
\draw (0,0) -- (-1.5,1.5)
(0,0) -- (1.5,-1.5);

\foreach \a in {(0,0), (1.5,0), (0,1.5), (-1.5,1.5), (-1.5,0), (0,-1.5), (1.5,-1.5)}
\fill[white] \a circle (0.05 cm);
\foreach \a in {(0,0), (1.5,0), (0,1.5), (-1.5,1.5), (-1.5,0), (0,-1.5), (1.5,-1.5)}
\draw \a+(0.05,0) arc (0:360:0.05);

\node at (1.7,1) {$\frakX_s$};
\end{tikzpicture}\end{center}
\end{ex}

%
%
%

The following proposition is the main result of this section.

\begin{prop}\label{prop_constructspecialfiber}
Suppose that $\Gamma$ is a smooth and 3-colorable tropical curve. Then there is a complete and connected curve $C_0\subseteq\frakX_s$ fulfilling the following properties:
\begin{enumerate}
\item For each vertex $v \in \Gamma$, let $C_v := C_0 \cap X_v \subseteq \frakX_s$; then $C_v \isom \PP^1$.
\item Each $C_v$ intersects every toric boundary stratum of $X_v$ transversally.
\item If $v$ and $w$ are two vertices of $\Gamma$ connected by an edge, then the components $C_v$ and $C_w$ intersect in a node. 
\end{enumerate}
\end{prop}

\begin{proof}
Let $\{v_i\}_i$ be an ordering of the vertices of $\Gamma$ as in Definition~\ref{definition_tame_tropical_curve}.
For every $i$, we inductively define a smooth rational curve $C_{v_i}$ in $P_{d}\times\{1\}\subseteq X_{v_i} \isom P_{d}\times \mathbb G_m^{n-d}$ subject to the following condition: if $v_j$ is a vertex of $\Gamma$ adjacent to ${v_i}$, for $j<i$, the two curves $C_{v_i}$ and $C_{v_j}$ intersect in a point of $X_{v_i}\cap X_{v_j}$. Such a curve $C_{v_i}$ exists since the condition that we are imposing is the passage through at most two given points.
Moreover, we can choose the curves $C_{v_i}$ meeting transversally each boundary stratum of $P_d\times\{1\}$, since a generic line in $\PP^d$ intersects any coordinate hyperplane transversally away from the strata of codimension two or higher.
Finally, let $C_0$ be the union of the curves $C_{v_i}$.
\end{proof}

\begin{ex}
Consider the tropical curve $\Gamma$ and the associated toric scheme $\frakX$ as in Example \ref{example_hexagonalthing}. Then the curve $C_0$ in $\frakX_s$ constructed in Proposition \ref{prop_constructspecialfiber} is a loop consisting of six copies of $\PP^1_{\C}$, and it can be visualized using tropical lines in the moment polytopes as follows: 
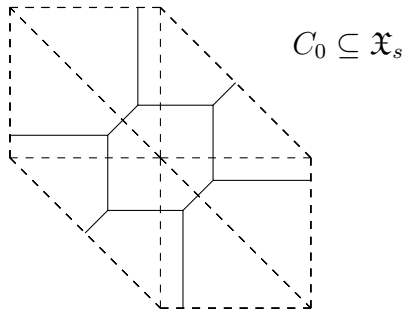
\vspace{1ex}
\begin{center}\begin{tikzpicture}

\foreach \a in {90,180,270,360}
\draw[dashed] (0,0) -- (\a:2)
 (2,0) -- (0,2)
 (0,2) -- (-2,2)
 (-2,2) -- (-2,0)
 (-2,0) -- (0,-2)
 (0,-2) -- (2,-2)
 (2,-2) -- (2,0)
 (0,0) -- (-2,2)
 (0,0) -- (2,-2);
 
 \draw (-0.3,0.7) -- (-0.3,2)
  (-0.3,0.7) -- (0,0.7)
  (-0.3,0.7) -- (-0.7,0.3);
  
  \draw (-0.7,0.3) -- (-2,0.3)
  (-0.7,0.3) -- (-0.7,0);
  
  \draw (0.7,0.7) -- (1,1)
  (0.7,0.7) -- (0,0.7)
  (0.7,0.7) -- (0.7,0);
  
   \draw (0.3,-0.7) -- (0.3,-2)
  (0.3,-0.7) -- (0,-0.7)
  (0.3,-0.7) -- (0.7,-0.3);

 \draw (-0.7,-0.7) -- (-1,-1)
  (-0.7,-0.7) -- (0,-0.7)
  (-0.7,-0.7) -- (-0.7,0);
  
 \draw (0.7,-0.3) -- (2,-0.3)
  (0.7,-0.3) -- (0.7,0);
  
\node at (2.5,1.5) {$C_0\subseteq \frakX_s$};
\end{tikzpicture}\end{center}
\end{ex}

\begin{rem}
Note that the curve $C_0$ constructed in Proposition \ref{prop_constructspecialfiber} is a nodal curve, and its dual graph is equal to the graph underlying the skeleton of $\Gamma$. We remind the reader that the data of an $n$-dimensional toric scheme over $R$ is essentially equivalent to the notion of a \emph{toric degeneration}, a toric morphism from a complex toric variety of dimension $n+1$ to $\mathbb{A}^1_\C$ as in \cite{NishinouSiebert06}*{Section 3}, and the embedding $C_0\subseteq \frakX_s$ is a pre-log curve in the sense of \cite{NishinouSiebert06}*{Definition 4.3}.
\end{rem}


\section{Log smooth deformation theory}\label{section_logdef}

In this section, we explain the conditions under which we can use log smooth deformation theory to lift the nodal curve $C_0$ to a semistable curve $\calC$ over $R$ in $\frakX$. Our approach is a generalization of the methods developed in \cite{NishinouSiebert06}.

We use logarithmic geometry in the sense of \cite{Kato89}, a theoretical framework that makes it possible to treat certain singularities, such as toric or normal crossings singularities, as if they were smooth. 
For the basics of this theory, see \cite{Kato89} and \cite{Gross_Kansaslectures}*{Chapter 3}. 

In our setting, we endow the scheme $O=\Spec R$ with the divisorial log structure defined by its special fiber. Its generic fiber is then $\Spec K$ with the trivial log structure, while its special fiber is the standard log point $O_0 := (\Spec\C,\C^{\times} \times \N)$.
We endow a toric scheme $\frakX$ with the divisorial log structure defined by its toric boundary; then $\frakX$ is log smooth over $O$. 
If $Y\to X$ is a morphism of log schemes, we denote by $\Theta_{Y/X}$ the \defi{log tangent sheaf} of $Y$ over $X$. 
By \cite{Oda88}*{Proposition 3.1}, there is then a natural isomorphism $\Theta_{\frakX/O}\isom \calO_{\frakX}\otimes_{\Z} N$.

Finally, we endow the nodal curve $C_0\subseteq \frakX_s$ with the log structure inherited from the log structure of $\frakX_s$ (see \cite{Kato89}*{Example 1.5(3)}). 
Then $C_0$ is log smooth over the standard log point $O_0$. 
Note that the log structure of $C_0$ not only encodes information about the nodal points, but also marks the points of the intersection of $C_0$ with those toric boundary divisors of $X_0$ which lie in only one component of $X_0$. 
In the situation of Section \ref{section_specialfiber}, this means that $C_0$ has one marked point for each unbounded edge of $\Gamma$.

We now develop the log smooth deformation theory we need for our proof.
Let $R_k = \C\lbrack\lbrack t^\frac{1}{\ell}\rbrack\rbrack/(t^\frac{k+1}{\ell})$ and endow $O_k=\Spec R_k$ with the log structure induced by $\N\to R_k:a\mapsto t^\frac{a}{\ell}$.
Note that we have natural closed immersions $O_{k'}\hookrightarrow O_k$ for $0\leq k'\leq k$.

\begin{defn}
Let $C_0$ be a log smooth curve over $O_0$. 
A \defi{$k$-th order deformation} of $C_0$ is a log smooth morphism $C_k\rightarrow O_k$ whose base change to $O_0$ is $C_0\rightarrow O_0$.
\end{defn}

Suppose we are given a $(k-1)$-st order deformation $C_{k-1}\rightarrow O_{k-1}$ of $C_0\rightarrow O_0$. 
By \cite{Gross_Kansaslectures}*{Proposition 3.40} there is an element $\ob(C_{k-1}/O_{k-1}) \in \HH^2\big(C_0,\Theta_{C_0/O_0}\big)$ such that $C_{k-1}\rightarrow O_{k-1}$ lifts to a $k$-th order deformation $C_k\rightarrow O_k$ if and only if $\ob(C_{k-1}/O_{k-1})=0$. 
Since $C_0$ is of dimension one, we have $\HH^2\big(C_0,\Theta_{C_0/O_0}\big)=0$ and therefore such a lift always exists. 
Moreover, the set of such lifts is a torsor over $\HH^1\big(C_0,\Theta_{C_0/O_0}\big)$. 
However, lifting a log smooth curve together with its embedding into $\frakX$ is more complicated.

Let $f=f_0:C_0\hookrightarrow\frakX_s$ be a strict closed immersion of $C_0$ into the special fiber $\frakX_s$ of a toric scheme $\frakX$, and consider the commutative diagram
\begin{equation*}\label{equation_specialfiber}\begin{CD}
C_0@>f_0>> \frakX\\
@VVV @VVV\\
O_0@>>> O 
\end{CD}\end{equation*}

\begin{defn} 
A \defi{$k$-th lift} of $f_0$ is a commutative diagram
\begin{equation*}\label{equation_kthlift}\begin{CD}
 C_k@>f_k>> \frakX\\
 @VVV @VVV \\
 O_k@>>> O
\end{CD}\end{equation*}
where $f_k:C_k\rightarrow \frakX$ is strict and $C_k\rightarrow O_k$ is a $k$-th lift of $C_0\rightarrow O_0$. In this case Nakayama's lemma guarantees that $f_k$ is a closed immersion. 
\end{defn}

The following proposition is the main result of this section; it expands on the argument in \cite{Gross_Kansaslectures}*{Theorem 3.41}. We refer the reader to \cite{NishinouSiebert06}*{Lemma 7.2 and Proposition 7.3} for the original results.

\begin{prop}
\label{P: existence deformation}
Let $f_{k-1}$ be a $(k-1)$-st lift of $f_0$ and suppose that the canonical homomorphism
\begin{equation}\label{equation_logtanemb}
\HH^1(C_0,\Theta_{C_0/O_0})\longrightarrow \HH^1(C_0,f_0^\ast\Theta_{\frakX/O})
\end{equation}
is surjective. Then there exists a $k$-th lift of $f_0$ that extends $f_{k-1}$. 
\end{prop}

In \cite{NishinouSiebert06} the authors assume that $C_0$ is rational. In this case $\HH^1(C_0,f_0^\ast\Theta_{\frakX/O})=0$ and the homomorphism \eqref{equation_logtanemb} is always surjective. 

\begin{proof}[Proof of Proposition \ref{P: existence deformation}]
Assume that we are given a $(k-1)$-st lift $f_{k-1}$ and let $C_k\rightarrow O_{k}$ be a lift of $C_{k-1}\rightarrow O_{k-1}$.

Suppose that \eqref{equation_logtanemb} is surjective. 
Choose an affine open cover $(U_i)$ of $C_0$ and let $U_{ij} = U_i \cap U_j$. Since the $U_i$ are affine, log smooth thickenings exist and are unique by \cite{Gross_Kansaslectures}*{Proposition 3.38}. Let $U_i^k$ and $U_i^{k-1}$ be log smooth thickenings of $U_i$ over $O_k$ and $O_{k-1}$ respectively.  In particular, we assume that $U_i^k$ is a lifting of $U_i^{k-1}$ to a log smooth scheme over $O_k$.
We have gluing morphisms 
\begin{equation*}
\theta_{ij}^k:U_{ij}^k\longrightarrow U_{ji}^k
\end{equation*}
and 
\begin{equation*}
\theta_{ij}^{k-1}:U_{ij}^{k-1}\longrightarrow U_{ji}^{k-1}
\end{equation*}
that fulfill $\theta_{ij}^{k-1} = \theta_{ij}^k|_{U_{ij}^{k-1}}$ for all $i$ and $j$.
Moreover, we have lifts $f_i^{k-1}:U_i^{k-1}\rightarrow \frakX$ that satisfy the compatibility condition $f_i^{k-1}=f_j^{k-1}\circ\theta_{ij}^{k-1}$ on $U_{ij}^{k-1}$.
 
Since $\frakX$ is log smooth over $O$, we can find lifts $f_i^k:U_i^k\rightarrow\frakX$ of the $f_i^{k-1}$ to the thickening $U_i^k$ of $U_i^{k-1}$ for every $i$. 
By \cite{Kato89}*{Proposition 3.9} the set of such lifts $f^k_i$ on $U_i^k$ forms a torsor over $\HH^0\big(U_i,f^\ast\Theta_{\frakX/O}\vert_{U_i}\big)$.

Now compare the two liftings $f_i^k$ and $f_j^k\circ \theta_{ij}^{k-1}$ on $U_{ij}^k$. 
Note that they both lift $f^{k-1}_i=f^{k-1}_j\circ\theta^{k-1}_{ij}$ on $U^{k-1}_{ij}$ and therefore differ by a section $\psi_{ij}\in\HH^0(U_{ij},f^\ast\Theta_{\frakX/O}\vert_{U_{ij}})$, i.e. we have
\begin{equation*}
f_i^k=f_j^k\circ\theta_{ij}^k + t^\frac{k}{N}\psi_{ij} \ .
\end{equation*}
The $\psi_{ij}$ define a $2$-cocycle of $f^\ast\Theta_{\frakX/O}$ on $C_0$, since $\HH^2(f^\ast\Theta_{\frakX/O})=0$. 

Since \eqref{equation_logtanemb} is surjective, there is a $2$-cocycle $\phi_{ij}^k$ for $\Theta_{C_0/O_0}$ such that $f\circ\phi_{ij}=\psi_{ij}$ for all $i$ and $j$. 
We can now replace the lift $C_k$ of $C_{k-1}$ by the lift $\tilde{C}_k$ of $C_{k-1}$ that is given by the gluing maps $\tilde{\theta}^k_{ij}=\theta^k_{ij} +t^\frac{k}{N}\phi_{ij}$. 
But then we have 
\begin{equation*}\begin{split}
f_i^k&=f_j^k\circ\theta_{ij}^k + t^\frac{k}{N}\psi_{ij}\\
&=f_j^k\circ(\tilde{\theta}_{ij}^k -t^\frac{k}{N}\phi_{ij}) + t^\frac{k}{N}\psi_{ij}\\
&=f_j^k\circ\tilde{\theta}_{ij}^k - t^\frac{k}{N}\psi_{ij} + t^\frac{k}{N}\psi_{ij} = f_j^k\circ\tilde{\theta}_{ij}^k
\end{split}\end{equation*}
and therefore we can glue the local lifts $f_i^k$ to a global lift $f_k:\tilde{C}_k\rightarrow\frakX$.
\end{proof}

In our setting, we can deduce the following result:

\begin{prop}\label{prop_constructmodel}
Let $\Gamma$, $\frakX$, and $C_0$ be as in Section \ref{section_specialfiber}, and assume that the homomorphism \eqref{equation_logtanemb} is surjective. Then there exists a semistable curve $\calC$ over $R$ with smooth generic fiber $C$ and special fiber $C_0$, together with a closed immersion $\calC \hookrightarrow \frakX$ extending $C_0\hookrightarrow\frakX$.
\end{prop}

\begin{proof}
Taking the direct limit of all $C_k$, we obtain a formal scheme $\frakC$ over $O$, with a closed immersion into the $t$-adic formal completion $\widehat{\frakX}$ of $\frakX$. Since $C_0$ is complete, all $C_k$ are proper over $O_k$ and therefore $\frakC$ is proper over $O$. By Grothendieck's existence theorem \cite{EGA3.1}*{Th\'eor\`eme 5.1.4}, $\frakC$ is then the $t$-adic formal completion of a closed subscheme $\calC$ of $\frakX$, proper over $O$.
Note that by construction the special fiber of $\calC$ is equal to $C_0$. 

We want to show that with the log structure induced from $\frakX$ the $R$-scheme $\calC$ is log smooth over $O$. 
Since this can be checked in an \'etale neighborhood in $\calC$ of a node $p$ of $C_0$, without loss of generality we can assume that $C_0=\Spec\big(\C[x,y]/(xy)\big)$.
For $e>0$, set $\calC_e=\Spec\big(R[x,y]/(xy-t^{{e}/{\ell}})\big)$ with log smooth logarithmic structure as in \cite{Gross_Kansaslectures}*{Example 3.26}.
By the description of log smooth curves of \cite{Kato_logsmoothcurves}*{Proposition 1.1}, there exists some $e>0$ such that the special fiber of $\calC_e$ is $C_0$. Therefore, for every $k>0$ the restriction $\calC_e\times_{O}O_k$ is the unique log smooth lifting of $C_0\to O_0$ to $O_k$. This implies that $\calC=\calC_e$, so $\calC$ is log smooth over $O$.

Then the generic fiber $C$ of $\calC$ is log smooth over $K$. Therefore $C$ has only toric singularities, hence it is smooth, since it is one-dimensional. Since $\calC$ is proper over $O$, the curve $C$ is complete.
\end{proof}

\begin{rem}\label{remark_markedpoints}
Assume we are in the situation of Proposition \ref{prop_constructmodel}. Then the log structure of $\calC$, which is the one induced by the log structure of $\frakX$, contains information not only about the nodes of $C_0$ but also about finitely many sections of $\calC\to O$, which are disjoint since $\calC$ is log smooth. In the special fiber, these sections cut out precisely the marked points of $C_0$, which correspond to the unbounded edges of $\Gamma$. On the other hand, in the generic fiber they cut out a finite set $V\subseteq C(K^{\textup{alg}})$ of marked points of $C$, which is precisely the intersection of $C$ with the toric boundary of $X$.
Therefore, this construction naturally gives rise to a generalized skeleton $\Sigma_{\calC,V}$ of $C^{an}$.
\end{rem}


\section{The abundancy map in cohomology}\label{section_abundancy}

Let $\Gamma\subseteq N_\R$ be a tropical curve with skeleton $G$, and denote by $E_G$ the set of  edges of $G$. 

Choose a direction for every edge $e \in E_G$ and write $\vec{e}$ for the vector in $N_{\R}$ connecting the two endpoints of $e$ according to this direction.

We denote by $\HH_1(\Gamma)$ the first simplicial homology group of $\Gamma$. 
An element of $\HH_1(\Gamma)$ is a formal sum $\sum_{e\in E_G} a_e [e]$, with integer coefficients, forming a cycle in $\Gamma$. 
The next definition is due to Mikhalkin \cite{Mikhalkin05}*{Section 2.6}, but our formulation is essentially the one of \cite{Katz12lifting}*{Section 1}.

\begin{defn}
\label{D: abundancy}
A tropical curve $\Gamma$ is said to be \emph{non-superabundant}, if the \defi{abundancy map} 
\begin{align*}
\Phi_\Gamma\colon \R^{E_G} &\longrightarrow \Hom\big(\HH_1(\Gamma), N_\R\big)\\
(\ell_e) &\longmapsto \left(\sum_{e \in E_G} a_e [e] \mapsto \sum_{e\in E_G} \ell_e a_e \vec{e} \right),
\end{align*}
is surjective.
\end{defn}

We remark that the homomorphism $\Phi_\Gamma\big((l_e)_e\big)$ does not depend on the choice of a direction of the edges of $\Gamma$, as edge directions are specified in both the source and the target of $\Hom\big(\HH_1(\Gamma), N_\R\big)$.

The intersection of the kernel of $\Phi_\Gamma$ with $\R^{E_G}_{>0}$ can be seen as the moduli space of metric graphs embedded in $N_\R$ which have the same combinatorial type as the skeleton of $\Gamma$, modulo translations. 
Then for $\Gamma$ to be non-superabundant means that this moduli space has the expected dimension $\#E_G-b_1(\Gamma)n$, where $b_1(\Gamma)$ is the first Betti number of $\Gamma$ and $n$ is the rank of $N$.
See \cite{Mikhalkin05}*{Sections 2.4--2.6} for a thorough discussion of those dimension counts. 
For example, whenever $\Gamma$ is trivalent, by \cite{Mikhalkin05}*{Proposition 2.13} the expected dimension of the moduli space of metric graphs in $N_\R$ which have the same combinatorial type as $\Gamma$ is $x+(n-3)(1-b_1(\Gamma))$, where $x$ is the number of unbounded edges of $\Gamma$.

Now let $\Gamma$ be a tropical curve, let $\frakX$ be the toric scheme as constructed in Section \ref{section_specialfiber}, and let $C_0$ be a nodal curve in $\frakX_s$ fulfilling the conclusion of Proposition \ref{prop_constructspecialfiber}. 
The crucial result of this section is the following proposition, which gives a cohomological interpretation of the abundancy map.

\begin{prop}\label{prop_cohsurj}
There are a homomorphism $\delta:\C^{E_G}\longrightarrow\HH^1(C_0,\Theta_{C_0/O_0})$, and an isomorphism
$\HH^1(C_0,f^\ast\Theta_{\frakX/O})\isom\Hom\big(\HH_1(\Gamma),N_\C\big)$
such that the induced homomorphism 
\begin{equation*}
\C^{E_G}\stackrel{\delta}{\longrightarrow}\HH^1(C_0,\Theta_{C_0/O_0})\longrightarrow \HH^1(C_0,f^\ast\Theta_{\frakX/O})\isom\Hom\big(\HH_1(\Gamma),N_\C\big)
\end{equation*}
is equal to $\Phi_\Gamma\otimes\C$. 
\end{prop}

The proof of this statement will require several steps.
We will begin by defining the homomorphism $\delta$.
In Lemma~\ref{lemma_simpcoh} we construct the isomorphism, and we conclude by explicitly computing the resulting composition.

\begin{proof}[Proof of Proposition \ref{prop_cohsurj}]
Note that there are compatible one-to-one correspondences between the vertices $v$ of $\Gamma$ and the components $C_v$ of $C_0$ as well as between the edges $e$ of $G$ and the corresponding nodes, denoted $p(e)$, of $C_0$. 
We have two normalization exact sequences (see \cite{Hartshorne77}*{Exercise IV.1.8}) of sheaves on $C_0$

\begin{equation}\label{sequences}
\xymatrix{
0 \ar[r] & \Theta_{C_0/O_0} \ar[r]\ar[d] & \prod_{v}(\Theta_{C_0/O_0})|_{C_v}\ar[r]\ar[d] &\prod_{e}\underline{\C}_{p(e)} \ar[r]\ar[d] & 0\\
0 \ar[r] & \calO_{C_{0}} \otimes N \ar[r] & \prod_{v}(\calO_{C_{0}/O_{0}} \otimes N)|_{C_v} \ar[r] & \prod_{e} \underline{N_\C}_{p(e)} \ar[r] & 0\\
}
\end{equation}
where $\underline{\C}_{p(e)}$ and $\underline{N_\C}_{p(e)}$ denote the skyscraper sheaves at $p(e)$. These products of the skyscraper sheaves are identified with the cokernel of the maps on the left of the horizontal exact sequences.

The first and the second vertical maps are given by composing the natural map $\Theta_{C_0/O_0}\rightarrow f^\ast\Theta_{\frakX/O}$ with the natural isomorphism $f^\ast\Theta_{\frakX/O}\isom\calO_{C_0}\otimes N$ of \cite{Oda88}*{Proposition 3.1}. 
These maps induce the third vertical map. 
Taking the long exact cohomology sequences of these two short exact sequences, we obtain the following commutative square:

\begin{equation}\label{sequences2}
\xymatrix{
\prod_{e}\C \ar[r]^-{\delta}\ar[d] & \HH^1( C_0, \Theta_{C_0/O_0}) \ar[d] \\
\prod_{e}N_\C \ar[r]^-{\delta} & \HH^1(C_0, \calO_{C_{0}} \otimes N) \\
}
\end{equation}

Now we need the following lemma:

\begin{lem}\label{lemma_simpcoh}
There is an isomorphism
\[
\alpha_0: \HH^1(C_0, \calO_{C_{0}}) \stackrel{\sim}{\longrightarrow} \Hom\big(\HH_1(\Gamma), \C).
\]
which induces the isomorphism
\[
\alpha: \HH^1(C_0, \calO_{C_{0}} \otimes N) \stackrel{\sim}{\longrightarrow} \Hom\big(\HH_1(\Gamma), N_\C\big).
\]
such that the composition $\alpha \circ \delta:\prod_{e}N_\C\rightarrow \Hom\big(\HH_1(\Gamma), N_\C\big)$ is given by sending a family $(u_e)_e$ of elements $u_e\in N_\C$ to the homomorphism 
\begin{equation*}
\sum_{e}a_e[e]\longmapsto \sum_{e}a_e u_e
\end{equation*}
in $\Hom\big(\HH_1(\Gamma),N_\C\big)$.
\end{lem}

\begin{proof}
Consider the normalization short exact sequence
\[
0 \to \calO_{C_0} \to \prod_{v} \calO_{C_v} \to \prod_{e} \underline{\C}_{p(e)} \to 0.
\]
The associated long exact cohomology sequence is
\begin{equation}\label{exactsequence}
0 \to \HH^0(C_0, \calO_{C_0}) \to \prod_{v} \HH^0(C_v, \calO_{C_v}) \to \prod_{e} \C \to \HH^1(C_0, \calO_{C_0}) \to 0,
\end{equation}
since by the rationality of $C_v$ we have $\HH^1(C_v, \calO_{C_v}) = 0$.

The first three terms of the sequence \eqref{exactsequence} are $\Hom(\cdot,\C)$ of the reduced simplicial chain complex 
\[
\Z^{E_G} \rightarrow \Z^{V_G} \rightarrow \Z \rightarrow 0,
\]
which defines $\HH_1(\Gamma)$.
Therefore we obtain an isomorphism $\HH^1(C_0, \calO_{C_0})\isom\HH^1(\Gamma, \C)$, and the latter is isomorphic to $\Hom\big(\HH_1(\Gamma),\C\big)$ by the universal coefficient theorem for $\Gamma$. Since $N$ is a free abelian group, we may apply $-\otimes N$ to $\alpha_0$, which induces the isomorphism $\alpha$. In this case, the homomorphism $\prod_{e}N_\C\rightarrow \Hom\big(\HH_1(\Gamma), N_\C\big)$ is given by sending a family $(u_e)_e$ of elements $u_e\in N_\C$ to the homomorphism 
\begin{equation*}
\sum_{e}a_e[e]\longmapsto \sum_{e}a_e u_e
\end{equation*}
in $\Hom\big(\HH_1(\Gamma),N_\C\big)$.
\end{proof}

Let us now finish the proof of Proposition~\ref{prop_cohsurj}. 
By Lemma \ref{lemma_simpcoh}, it is enough to show that the homomorphism 
\begin{equation}\label{eq_prodprod}
\prod_{e}\C\longrightarrow  \prod_{e}N_\C 
\end{equation}
on the left of diagram \eqref{sequences2} is given by sending $(l_e)_e$ to the family $(l_e \vec{e})_e$ in $\prod_{e}N_\C$. 

Let $e$ be an edge of $\Gamma$. Let $v_1$ and $v_2$ be the two vectors in $N_\R\times\R_{\geq 0}$ pointing to the two ends of $e$, so that $v_2-v_1=\vec{e}$.
We complete $\{v_1,v_2\}$ to a $\Z^n\oplus \frac{1}{\ell}\Z$-integral basis $\{v_1,\ldots, v_{n+1}\}$ of $N_\R\times\R_{\geq 0}$, and we write $x_1,\ldots, x_{n+1}$ for the induced coordinates on the open affine torus-invariant subset $\mathfrak{U}$ corresponding to the cone in $\Delta$ containing $e$.
Then $\Theta_{\frakU/O}$ has generators $x_1\frac{\partial}{\partial x_1}, \ldots, x_{n+1}\frac{\partial}{\partial x_{n+1}}$ that fulfill the relation $T\frac{\partial}{\partial T}=0$, where $T$ denote the image in $\calO_{\frakU}$ of the coordinate $t^{\frac{1}{\ell}}$ on $O$. 
Since $f_0\mathrel{\mathop:}C_0\hookrightarrow \frakX$ is a strict closed immersion, the pullbacks $y_1=f_0^\ast x_1$ and $y_2=f_0^\ast x_2$ define coordinates on $\calO_{C_0}$ such that the formal completion of $\calO_{C_0,p(e)}$ is isomorphic to $\mathbb{C}[[y_1y_2]]/(y_1y_2)$. 
In these coordinates the stalk $(\Theta_{C_0/O_0})_{p(e)}$ is generated by the two elements $y_1\frac{\partial}{\partial y_1}$ and $y_2\frac{\partial}{\partial y_2}$, which fulfill the relation $y_1\frac{\partial}{\partial y_1} +y_2\frac{\partial}{\partial y_2}=0$.

Therefore around $p(e)$ the natural homomorphism 
\begin{equation*}
\Theta_{C_0/O_0}\longrightarrow f^\ast\Theta_{\frakX/O}
\end{equation*}
is given by the associations $y_1\frac{\partial}{\partial y_1}\mapsto x_1\frac{\partial}{\partial x_1}$ and $y_2\frac{\partial}{\partial y_2}\mapsto x_2\frac{\partial}{\partial x_2}$. 
Since $v_2-v_1=\vec{e}$ the map \eqref{eq_prodprod} sends $(0,\ldots,0,l_e,0,\ldots,0)$ to $(0,\ldots,0,l_e \vec{e},0,\ldots,0)$, which concludes the proof of Proposition~\ref{prop_cohsurj}.

\end{proof}

\begin{rem}
In particular, if $\Gamma$ is trivalent, then the homomorphism \eqref{equation_logtanemb} is precisely $\Phi_\Gamma\otimes\C$. 
Indeed, in this case, for every vertex $v$, we have $\Theta_{C_0/O_0}|_{C_v}\isom\calO_{\PP^1}(-1)$, so $\HH^i(C_0,\Theta_{C_0/O_0}|_{C_v})=0$ for $i=0,1$, and therefore the homomorphism $\delta$ in Proposition \ref{prop_cohsurj} is an isomorphism by the long exact sequence associated to the first line of the diagram \eqref{sequences}. This is the case considered in \cite{Nis141}. 
\end{rem}


\section{Proofs of Theorems \ref{T: one} and \ref{T: two}}\label{section_proofs}

In this section we complete the proofs of Theorem~\ref{T: one} and Theorem~\ref{T: two}.
It is worthwhile to notice that 3-colorability and smoothness of the tropical curve $\Gamma\subseteq N_\R$ are only used to construct a suitable nodal curve $C_0\subseteq\frakX_s$ in Section~\ref{section_specialfiber}.
Whenever such a nodal curve exists, the non-superabundancy of $\Gamma$ is sufficient for the conclusion of Theorem~\ref{T: one} to hold. 

Let $\Gamma\subseteq N_\R$ be a tropical curve, let $\frakX$ be the toric $R$-scheme defined as in Section~\ref{section_specialfiber}, and denote by $X$ the generic fiber of $\frakX$.

\begin{lem}\label{lemma_trop=Gamma}
Let $\calC$ be a flat $R$-curve in $\frakX$ and denote by $C\subseteq X$ its generic fiber. 
If the special fiber $\mathcal{C}_s$ is proper over $\C$ and intersects every torus orbit contained in the special fiber of $\frakX$, then $\Trop(C\cap T)=\Gamma$.
\end{lem}

\begin{proof}
By \cite{Gubler13}*{Proposition 11.12} the properness of $\mathcal{C}_s$ implies that $\Trop(C)\subseteq \Gamma$, and therefore the vertices of $\Trop(C)$ are a subset of $\Gamma$. 
By Tevelev's Lemma \cite{Gubler13}*{Lemma 11.6} $\Trop(C)$ intersects the relative interior of every face in $\Gamma$. 
In particular, the vertices of $\Trop(C)$ coincide with the vertices of $\Gamma$. 
Again, since $\Trop(C)$ intersects the relative interior of every face of $\Gamma$, all one-dimensional faces of $\Gamma$ already have to be contained in $\Trop(C)$, and therefore $\Trop(C)=\Gamma$.
\end{proof}

In the proof of Theorem \ref{T: one} we use the \emph{initial degeneration} $\indeg_P(C)$ of $C$ along an open face $P$ of $\Trop(C)$. In our situation $\indeg_P(C)$ can be defined as the $\C$-scheme
\begin{equation*}
\indeg_P(C)=\big(\calC_s\cap \frakX_P\big)\times T_P \ ,
\end{equation*}
where $\frakX_P$ is the torus orbit in $\frakX_s$ corresponding to $P$, and $T_P$ is the subgroup of the reduction of the torus $\calT=\Spec R[M]$ that stabilizes $\frakX_P$. 
By \cite{HelmKatz12}*{Lemma 3.6} this coincides with the usual definition of the initial degeneration of $C$ at a point $p\in P$, as for example in \cite{BPR11}*{Section 2.1}, since $\calC$ is proper over $O=\Spec R$ and the multiplication map $\calT\times_O\calC\rightarrow\frakX$ is flat by the same argument as in \cite{Hacking08}*{Lemma 2.7} and surjective because $\calC$ meets each torus orbit of $\frakX$.

\begin{proof}[Proof of Theorem \ref{T: one}]
Let $\Gamma\subset N_\R$ be a non-superabundant, smooth and 3-colorable tropical curve with rational edge lengths, let $K$ and $\frakX$ be be defined as in Section \ref{section_specialfiber}, and $C_0$ be a curve in $\frakX_s$ as constructed in Proposition \ref{prop_constructspecialfiber}. 
Since $\Gamma$ is non-superabundant, Proposition \ref{prop_cohsurj} and Proposition \ref{prop_constructmodel} imply that there is a semistable curve $\calC$ over $R$ with smooth generic fiber $C$ and special fiber $C_0$, together with a closed immersion $\calC\hookrightarrow\frakX$. 
By Lemma \ref{lemma_trop=Gamma} we have $\Trop(C\cap T)=\Gamma$. 
By construction of $C_0$, all the initial degenerations of $C$ along open faces of $\Gamma$ are smooth and irreducible, and so the tropical multiplicities are all one. 
Therefore, \cite{BPR11}*{Corollary 6.11} implies that the tropicalization is faithful with respect to the skeleton $\Sigma_{\calC,V}$ of $C^{an}$, where $V=C\setminus(C\cap T)$ is the set of marked points as described in Remark~\ref{remark_markedpoints}.
\end{proof}

Theorem \ref{T: two} follows from Theorem \ref{T: one} and the following result, which is a slight extension of a theorem of Cartwright--Dudzik--Manjunath--Yao.

\begin{thm}[\cite{CDMY}]
\label{thm: combinatorics}
Let $G$ be a metric graph with rational edge lengths. Set 
\begin{equation*}
n = \max \big\{3, \max \{\deg v - 1 \mid v \in E_G\}\big\}
\end{equation*}
and let $N$ be a free abelian group of rank $n$. Then there exists a non-superabundant, smooth and 3-colorable tropical curve $\Gamma\subseteq N_{\R}$ with rational edge lengths and whose skeleton is isomorphic to $G$.
 \end{thm}

\begin{proof} Fix an integral basis $v_1,\ldots, v_n$ of $N_\R$. Let $\Gamma$ be a tropical curve as given by \cite{CDMY}*{Theorem 1.1}. As noted in \cite{CDMY}*{Remark 2.8}, the only part that is not included there is the non-superabundance of $\Gamma$.  
Let $G'$ be the minimal finite graph underlying the skeleton of $\Gamma$ and consider a spanning tree $\mathbb{T}$ of $G'$.
Then the set of $\epsilon\in E_{\tilde{G}\setminus\mathbb{T}}$ parametrizes a basis $\{c_\epsilon\}$ of $\HH_1(\Gamma,\mathbb Z)$, and $c_\epsilon$ is the only element of the basis which contains $\epsilon$.

In order to show the surjectivity of the abundancy map it is enough to show that for all $\epsilon\in E_{\tilde{G}\setminus\mathbb{T}}$ and $1\leq i\leq n$ the homomorphisms
\begin{equation*}
\begin{split}
f_{\epsilon,i}\mathrel{\mathop:}\HH_1(\Gamma)&\longrightarrow N_\R\\
c_{\epsilon'} &\longmapsto \begin{cases} v_i &\mbox{if } \epsilon' =\epsilon \\
0 & \mbox{if } \epsilon' \neq \epsilon. \end{cases}\end{split}
\end{equation*}
are in the image of $\Phi_\Gamma$. By the construction of \cite{CDMY}, every $\epsilon$ will contain an edge $e$ of $G'$ such that $\vec{e}$ is parallel to $v_i$. 
Since $c_\epsilon$ is the only element of the basis which contains $e$, if we take the vector $\ell\in\R^{\#E_{G'}}$ with value $|v_i|$ in the $e$-th entry and $0$ otherwise, we obtain $\Phi(\ell)=f_{\epsilon,i}$.
\end{proof}

\begin{rem}
It is possible to extend Theorem~\ref{T: one}, and therefore Theorem~\ref{T: two}, to the equicharacteristic $p >0$ case.
Let $\Gamma\subset N_\R$ be a non-superabundant, smooth and 3-colorable tropical curve. 
Then for all but finitely many prime numbers $p$ there exists a finite field extension $K$ of $\mathbb F_p((t))$ such that we can construct a suitable curve $C_{0}$ in the special fiber of $\frakX$ using the method of Section~\ref{section_specialfiber}.
The results of Section~\ref{section_logdef} remain valid over any discrete valuation ring of the form $k[[t]]$, where $k$ is an arbitrary field, so in particular they hold over $R$.
Observe that the abundancy map $\Phi_\Gamma$ defined in \ref{D: abundancy} is the base change $\Phi_{\Gamma,\Z}\otimes\R$ of the map $\Phi_{\Gamma,\mathbb Z}\colon \mathbb Z^{E_G} \to \Hom\big(\HH_1(\Gamma,\Z), N\big)$ defined by 
\begin{equation*}
(\ell_e) \longmapsto \Big(\sum_{e \in E_G} a_e [e] \mapsto \sum_{e\in E_G} \ell_e a_e \vec{e} \Big)  
\end{equation*}
as in Definition \ref{D: abundancy}. For a field $L$ we denote by $\Phi_{\Gamma,L}$ the base change of $\Phi_{\Gamma,\mathbb Z}$ to $L$.
Given two fields $L$ and $L'$ of the same characteristic, the surjectivity of $\Phi_{\Gamma,L}$ is equivalent to the surjectivity of $\Phi_{\Gamma,L'}$, and in particular $\Gamma$ is non-superabundant if and only if $\Phi_{\Gamma,\Q}$ is surjective. In this case the map $\Phi_{\Gamma,\mathbb F_p}$ is surjective for all but finitely many prime numbers $p$. 
Therefore, if $\Gamma$ is non-superabundant, log deformations of $C_{0}$ can be constructed for $p$ big enough and it follows that, given a tropical curve $\Gamma$ satisfying the hypotheses of Theorem~\ref{T: one},  for all but finitely many prime numbers $p$ we can find a finite field extension $K$ of $\mathbb F_p((t))$ such that the conclusion of Theorem~\ref{T: one} holds over $K$.
\end{rem}

\bibliographystyle{alpha}
\bibliography{biblio}{}

\end{document}